\documentclass[11pt]{article}
\usepackage{amsmath,amsfonts,amssymb,amsthm}
\usepackage{enumerate}
\usepackage{esint}
\usepackage[pagebackref,colorlinks,citecolor=blue,linkcolor=blue]{hyperref}
\usepackage[dvipsnames]{xcolor}
\usepackage{mathtools}  

\usepackage{comment}

\usepackage{geometry}
\numberwithin{equation}{section}
\theoremstyle{plain}
\newtheorem{theorem}[equation]{Theorem}
\newtheorem{corollary}[equation]{Corollary}
\newtheorem{lemma}[equation]{Lemma}
\newtheorem{proposition}[equation]{Proposition}

\theoremstyle{definition}
\newtheorem{definition}[equation]{Definition}

\newtheorem{example}[equation]{Example}
\newtheorem{remark}[equation]{Remark}

\numberwithin{equation}{section}

\newcommand{\Z}{{\mathbb Z}}
\newcommand{\R}{{\mathbb R}}
\newcommand{\N}{{\mathbb N}}

\newcommand{\Om}{\Omega}

\providecommand{\vint}[1]{\mathchoice
          {\mathop{\vrule width 5pt height 3 pt depth -2.5pt
                  \kern -9pt \kern 1pt\intop}\nolimits_{\kern -5pt{#1}}}
          {\mathop{\vrule width 5pt height 3 pt depth -2.6pt
                  \kern -6pt \intop}\nolimits_{\kern -3pt{#1}}}
          {\mathop{\vrule width 5pt height 3 pt depth -2.6pt
                  \kern -6pt \intop}\nolimits_{\kern -3pt{#1}}}
          {\mathop{\vrule width 5pt height 3 pt depth -2.6pt
                  \kern -6pt \intop}\nolimits_{\kern -3pt{#1}}}}

\newcommand{\eps}{\varepsilon}
\newcommand{\loc}{\mathrm{loc}}

\newcommand{\BV}{\mathrm{BV}}

\newcommand{\ch}{\text{\raise 1.3pt \hbox{$\chi$}\kern-0.2pt}}

\DeclareMathOperator{\Mod}{Mod}

\DeclareMathOperator{\dist}{dist}
\DeclareMathOperator{\diam}{diam}

\DeclareMathOperator{\Lip}{Lip}

\DeclareMathOperator{\supp}{spt}

\begin{document}
\title{A characterization of BV and Sobolev functions\\
	via nonlocal functionals in metric spaces
\footnote{
\hfill \break {\it Keywords\,}: Sobolev function, function of bounded variation,
metric measure space, nonlocal functional, Poincar\'e inequality
}}
\author{Panu Lahti, Andrea Pinamonti, and Xiaodan Zhou}

\maketitle

\begin{abstract}
We study a characterization of BV and Sobolev functions via nonlocal functionals in metric spaces equipped with a doubling
measure and supporting a Poincar\'e inequality.
Compared with previous works, we consider more general functionals.
We also give a counterexample in the case $p=1$ demonstrating that unlike in Euclidean spaces,
in metric measure spaces the limit of the nonlocal functions is only comparable, not necessarily equal,
to the variation measure $\Vert Df\Vert(\Om)$.
\end{abstract}

\section{Introduction}

Consider a sequence $\{\rho_i\}_{i=1}^{\infty}$ of nonnegative functions in $L^1(\R^n)$, $n\ge 1$,
which are radial (i.e. only depend on $|x|$) and for which
\begin{equation}\label{eq:intro rhoi conditions}
\int_{\R^n}\rho_i(x)\,dx=1\ \textrm{ for all }i\in\N
\quad\textrm{and}\quad
\lim_{i\to \infty}\int_{|x|>\delta}\rho_i(x)\,dx=0\quad\textrm{for all }\delta>0.
\end{equation}
For an open set $\Om\subset\R^n$ and function $f\in W_{\loc}^{1,p}(\Om)$, we define the Sobolev seminorm by
\[
|f|_{W^{1,p}(\Om)}:=\left(\int_{\Om}|\nabla f|^p\,dx\right)^{1/p},
\]
and if $f\notin W_{\loc}^{1,p}(\Om)$, then we let $|f|_{W^{1,p}(\Om)}=\infty$.
Bourgain, Brezis, and Mironescu \cite[Theorem 3]{BBM} showed that when $\Om\subset\R^n$ is a smooth, bounded domain and
$1<p<\infty$, then for every $f\in L^p(\Om)$ we have
\[
\lim_{i\to\infty}\int_{\Om}\int_{\Om}\frac{|f(x)-f(y)|^p}{|x-y|^p}\rho_i(|x-y|)\,dx\,dy=K_{p,n}|f|^p_{W^{1,p}(\Om)};
\]
here $K_{p,n}$ is a constant depending only on $p,n$.
D\'avila \cite{Dav} generalized this result to functions of bounded variation (BV functions)
$f$ and their variation measures $\Vert Df\Vert$.
He showed that when $\Om\subset\R^n$ is a bounded domain with Lipschitz boundary,
then for every $f\in L^1(\Om)$ we have
\[
\lim_{i\to\infty}\int_{\Om}\int_{\Om}\frac{|f(x)-f(y)|}{|x-y|}\rho_i(|x-y|)\,dx\,dy
=K_{1,n}\Vert Df\Vert(\Om),
\]
where we understand $\Vert Df\Vert(\Om)=\infty$ if $f\notin \BV(\Om)$.
To unify the notation, we define the energy
\[
E_{f,p}(\Om)
:=\begin{cases}
\Vert Df\Vert(\Om) &\textrm{when }p=1,\\
\int_{\Om}|\nabla f|^p\,dx &\textrm{when }1<p<\infty.
\end{cases}
\]
Several different generalizations of these results in Euclidean spaces
have been considered e.g.
by Ponce \cite{Pon}, Leoni--Spector \cite{LS1,LS2},
Brezis--Van Schaftingen--Yung \cite{Bre2,Bre3,Bre4},
Nguyen--Pinamonti--Vecchi--Squassina \cite{NPVS,PVS1},  Nguyen \cite{N1}, Brezis--Nguyen \cite{BN1,BN2, BN3},  Garofalo--Tralli \cite{GT, GT1} and Comi-Stefani \cite{CS,CS1,CS2}.

Brezis \cite[Remark 6]{Bre} suggested generalizing the theory to more general metric measure spaces $(X,d,\mu)$.
One generalization was given by Di Marino--Squassina \cite{DiMaS},
who assumed the measure $\mu$ to be doubling and the space to support a $(p,p)$--Poincar\'e inequality.
Such spaces are often called PI spaces.
We will give definitions in Section \ref{sec:definitions}.
They considered the mollifiers
\[
\rho_s(x,y):=(1-s)\frac{1}{d(x,y)^{ps}\mu(B(y,d(x,y)))},\quad x,y\in X,\quad 0<s< 1,
\]
and showed in \cite[Theorem 1.4]{DiMaS}
that for a constant $C\ge 1$ and for every $f\in L^p(X)$, we have
\begin{equation}\label{eq:DiMarinoSquassina}
\begin{split}
&C^{-1}E_{f,p}(X)
\le \liminf_{s\nearrow 1}(1-s)\int_{X}\int_{X}\frac{|f(x)-f(y)|^p}{d(x,y)^{ps}\mu(B(y,d(x,y)))}\,d\mu(y)\,d\mu(x)\\
&\qquad \le \limsup_{s\nearrow 1}(1-s)\int_{X}\int_{X}\frac{|f(x)-f(y)|^p}{d(x,y)^{ps}\mu(B(y,d(x,y)))}\,d\mu(y)\,d\mu(x)
\le CE_{f,p}(X).
\end{split}
\end{equation}
A similar result was proved previously in Ahlfors-regular spaces in \cite{Mun}.
G\'orny \cite{Gor}, resp. Han--Pinamonti \cite{HP}, studied the problem in certain PI spaces that ``locally
look like'' Euclidean spaces, resp. finite-dimensional Banach spaces or Carnot groups,
and showed that for every $f\in N^{1,p}(X)$, with $1<p<\infty$, we have
\[
\lim_{r\to 0}\frac{1}{r^p}\int_X \vint{B(y,r)}|f(x)-f(y)|^p\,d\mu(x)\,d\mu(y)=C E_{f,p}(X).
\]
These results correspond to certain choices of the mollifiers $\rho_i$ satisfying
\eqref{eq:intro rhoi conditions}.
In the current paper, our main goal is to study this problem
for more general mollifiers $\rho_i$, of which the mollifiers considered in
\cite{DiMaS,Gor,Mun} are special cases.
Moreover, we consider domains $\Om\neq X$.

Our main result is the following.

\begin{theorem}\label{thm:main}
	Let $1\le p<\infty$, and suppose $\mu$ is doubling and $X$ supports a $(p,p)$--Poincar\'e inequality.
	Suppose $\{\rho_i\}_{i=1}^{\infty}$ is a sequence of mollifiers satisfying conditions
	\eqref{eq:rho hat minorize}--\eqref{eq:radius one condition}.
	Suppose $\Om\subset X$ is a strong $p$-extension domain, and let $f\in L^p(\Om)$. Then
	\begin{equation}\label{eq:main theorem equation}
	\begin{split}
	C_1 E_{f,p}(\Om)
	&\le \liminf_{i\to\infty}\int_{\Om}\int_{\Om} \frac{|f(x)-f(y)|^p}{d(x,y)^p}\rho_i(x,y)\,d\mu(x)\,d\mu(y)\\
	&\le \limsup_{i\to\infty}\int_{\Om}\int_{\Om} \frac{|f(x)-f(y)|^p}{d(x,y)^p}\rho_i(x,y)\,d\mu(x)\,d\mu(y)
	\le C_2 E_{f,p}(\Om)
	\end{split}
	\end{equation}
	for some constants $C_1\le C_2$ that depend only on $p$,
	the doubling constant of the measure, the constants in the
	Poincar\'e inequality, and the constant $C_{\rho}$ associated with the mollifiers.
\end{theorem}

After giving definitions in Section \ref{sec:definitions} and some preliminary results in
Section \ref{sec:prelis}, we prove the two directions of  \eqref{eq:main theorem equation}
in Sections \ref{sec:one direction} and \ref{sec:other direction}.
In Section \ref{sec:quasidecreasing} we give corollaries to our main Theorem
\ref{thm:main}, showing that the mollifiers considered in \cite{DiMaS} and \cite{Gor},
as well as other choices, can be handled as special cases.
In Section \ref{sec:counterexample} we give a counterexample demonstrating that
we do not generally have $C_1=C_2$ in \eqref{eq:main theorem equation}.

\vspace{1cm}

\textbf{Acknowledgement:} The authors would like to thank Camillo Brena and Enrico Pasqualetto for some useful comments on a preliminary version of the paper.

\section{Notation and definitions}\label{sec:definitions}

Throughout this paper, we work in a complete and connected
metric measure space $(X,d,\mu)$ equipped with a metric $d$ and
a Borel regular outer measure $\mu$ satisfying
a doubling property, meaning that
there exists a constant $C_d\ge 1$ such that
\[
0<\mu(B(x,2r))\le C_d\mu(B(x,r))<\infty
\]
for every ball $B(x,r):=\{y\in X\colon\,d(y,x)<r\}$.
We assume that  $1\le p<\infty$ and $X$ consists of at least two points, that is, $\diam X>0$.

By a curve we mean a rectifiable continuous mapping from a compact interval of the real line into $X$.
The length of a curve $\gamma$
is denoted by $\ell_{\gamma}$. We will assume every curve to be parametrized
by arc-length, which can always be done (see e.g. \cite[Theorem~3.2]{Hj}).
A nonnegative Borel function $g$ on $X$ is an upper gradient 
of a function $f\colon X\to [-\infty,\infty]$
if for all nonconstant curves $\gamma\colon [0,\ell_{\gamma}]\to X$, we have
\begin{equation}\label{eq:definition of upper gradient}
	|f(x)-f(y)|\le \int_{\gamma} g\,ds:=\int_0^{\ell_{\gamma}} g(\gamma(s))\,ds,
\end{equation}
where $x$ and $y$ are the end points of $\gamma$.
We interpret $|f(x)-f(y)|=\infty$ whenever  
at least one of $|f(x)|$, $|f(y)|$ is infinite.
Upper gradients were originally introduced in \cite{HK}.

We always consider $1\le p<\infty$.
The $p$-modulus of a family of curves $\Gamma$ is defined by
\[
\Mod_{p}(\Gamma):=\inf\int_{X}\rho^p\, d\mu,
\]
where the infimum is taken over all nonnegative Borel functions $\rho$
such that $\int_{\gamma}\rho\,ds\ge 1$ for every curve $\gamma\in\Gamma$.
A property is said to hold for $p$-almost every curve
if it fails only for a curve family with zero $p$-modulus. 
If $g$ is a nonnegative $\mu$-measurable function on $X$
and (\ref{eq:definition of upper gradient}) holds for $p$-almost every curve,
we say that $g$ is a $p$-weak upper gradient of $f$. 
By only considering curves $\gamma$ in a set $A\subset X$,
we can talk about a function $g$ being a ($p$-weak) upper gradient of $u$ in $A$.\label{curve discussion}

We always let $\Om$ denote an open subset of $X$.
We define the Newton-Sobolev space $N^{1,p}(\Om)$ to consist of those functions $f\in L^p(\Om)$ for which there
exists  a $p$-weak upper gradient $g\in L^p(\Om)$ of $f$ in $\Om$.
This space was first introduced in \cite{S}.
We write $f\in N^{1,p}_{\loc}(\Om)$ if for every $x\in \Om$ there exists $r>0$ such that
$f\in N^{1,p}(B(x,r))$; other local function spaces are defined analogously.
For every $f\in N^{1,p}_{\loc}(\Om)$ there exists a minimal $p$-weak
upper gradient of $f$ in $\Om$, denoted by $g_f$, satisfying $g_f\le g$ 
$\mu$-almost everywhere (a.e.)
in $\Om$ for every $p$-weak upper gradient $g\in L_{\loc}^{p}(\Om)$
of $f$ in $\Om$, see \cite[Theorem 2.25]{BB}.

Note that Newton-Sobolev functions are understood to be defined at every $x\in \Om$, whereas the functionals
that we consider are not affected by perturbations of $f$ in a set of zero $\mu$-measure. For this reason, we also define
\[
\widehat{N}^{1,p}(\Om):=\{f\colon f=h \ \mu\textrm{-a.e. in }\Om\textrm{ for some }h\in N^{1,p}(\Om)\}.
\]
For every $f\in \widehat{N}^{1,p}(\Om)$, we can also define
$g_f:=g_h$, where $g_h$ is the minimal $p$-weak upper gradient of any $h$ as above in $\Om$;
this is well defined $\mu$-a.e. in $\Om$ by \cite[Corollary 1.49, Proposition 1.59]{BB}.

Next we define functions of bounded variation.
Given an open set $\Om\subset X$ and a function $f\in L^1_{\loc}(\Om)$,
we define the total variation of $f$ in $\Om$ by
\[
\Vert Df \Vert(\Om):=\inf\left\{\liminf_{i\to\infty}\int_\Om g_{f_i}\,d\mu:\, f_i\in N^{1,1}_{\loc}(\Om),\, f_i\to f\textrm{ in } L^1_{\loc}(\Om)\right\},
\]
where each $g_{f_i}$ is the minimal $1$-weak upper gradient of $f_i$
in $\Om$.
We say that a function $f\in L^1(\Om)$ is of bounded variation, 
and denote $f\in\BV(\Om)$, if $\Vert Df\Vert (\Om)<\infty$.
For an arbitrary set $A\subset X$, we define
\[
\Vert Df \Vert (A):=\inf\{\Vert Df\Vert (W):\, A\subset W,\,W\subset X
\text{ is open}\}.
\]
If $f\in \BV_{\loc}(\Om)$,
then $\Vert Df\Vert(\cdot)$ is
a Radon measure on $\Omega$ by \cite[Theorem 3.4]{Mir}.

Next we record
Mazur's lemma and Fuglede's lemma, see e.g.
\cite[Theorem 3.12]{Rud}
and \cite[Lemma 2.1]{BB}.

\begin{theorem}\label{thm:Mazur lemma}
	Let $\{g_i\}_{i=1}^{\infty}$ be a sequence
	with $g_i\to g$ weakly in $L^p(\Om)$.
	Then there exist convex combinations $\widehat{g}_i:=\sum_{j=i}^{N_i}a_{i,j}g_j$,
	for some $N_i\in\N$,
	such that $\widehat{g}_i\to g$ in $L^p(\Om)$.
\end{theorem}

By convex combinations we mean that the numbers $a_{i,j}$ are nonnegative
and that $\sum_{j=i}^{N_i}a_{i,j}=1$ for every $i\in\N$.

\begin{lemma}\label{lem:Fuglede lemma}
	Let $\{g_i\}_{i=1}^{\infty}$ be a sequence of functions with $g_i\to g$ in $L^p(\Om)$.
	Then for $p$-a.e. curve $\gamma$ in $\Om$, we have
	\[
	\int_{\gamma}g_i\,ds\to \int_{\gamma}g\,ds\quad\textrm{as }i\to\infty.
	\]
\end{lemma}

We say that $X$ supports a $(p,p)$-Poincar\'e inequality,
if there exist constants $C_P>0$ and $\lambda \ge 1$ such that for every
ball $B(x,r)$, every $f\in L^p(X)$,
and every $p$-weak upper gradient $g$ of $f$, we have
\begin{equation}\label{eq:pp Poincare}
\int_{B(x,r)}|f-f_{B(x,r)}|^p\, d\mu \le C_P r^p \int_{B(x,\lambda r)}g^p\,d\mu,
\end{equation}
where 
\[
f_{B(x,r)}:=\vint{B(x,r)}f\,d\mu :=\frac 1{\mu(B(x,r))}\int_{B(x,r)}f\,d\mu.
\]
In the case $p=1$, the following BV version of the Poincar\'e inequality can be obtained
by applying the $(1,1)$--Poincar\'e inequality to the approximating functions in the definition
of the total variation: for every $f\in L^1(X)$, we have
\begin{equation}\label{eq:poincare inequality BV}
\int_{B(x,r)}|f-f_{B(x,r)}|\, d\mu \le C_P r \Vert Df\Vert(B(x,\lambda r)).
\end{equation}
Suppose $f\in\BV(\Om)$ if $p=1$, and $f\in \widehat{N}^{1,p}(\Om)$ if $1<p<\infty$.
In the latter case, denote the minimal $p$-weak upper gradient of $f$ in $\Om$ by $g_f$.
For every Borel set $A\subset \Om$, we denote the energy by
\[
E_{f,p}(A)
:=\begin{cases}
\Vert Df\Vert(A)\quad\textrm{when }p=1\\
\int_{A}g_f^p\,d\mu \quad\textrm{when }1<p<\infty.
\end{cases}
\]
Note that $E_{f,p}$ is then a Borel measure on $\Om$.
If $f$ is not in $\BV(\Om)$ (in the case $p=1$),
respectively not in $\widehat{N}^{1,p}(\Om)$ (in the case $1<p<\infty$), then we let $E_{f,p}(\Om)=\infty$.
We can combine \eqref{eq:pp Poincare} and \eqref{eq:poincare inequality BV} to give: for every
$1\le p <\infty$ and every $f\in L^p(X)$, we have
\begin{equation}\label{eq:poincare general form}
\int_{B(x,r)}|f-f_{B(x,r)}|^p\, d\mu 
\le C_P r^p E_{f,p}(B(x,\lambda r)).
\end{equation}

\begin{definition}\label{def:strong BV extension}
We say that an open set $\Om\subset X$ is a strong $p$-extension domain if
\begin{itemize}
	\item in the case $p=1$, for every $f\in \BV(\Om)$
	there exists an extension $F\in \BV(X)$;
	\item in the case $1<p<\infty$, for every $f\in \widehat{N}^{1,p}(\Om)$ there exists
	an extension $F\in \widehat{N}^{1,p}(X)$;
\end{itemize}
and in both cases, $E_{F,p}(\partial\Om)=0$.
\end{definition}

For example, in Euclidean spaces, a bounded domain with a Lipschitz boundary is a strong $p$-extension domain for all $1\le p<\infty$,
see e.g. \cite[Proposition 3.21]{AFP}.

Now we describe the mollifiers that we will use.
We will consider a sequence of nonnegative $X\times X$-measurable functions
$\{\rho_i(x,y)\}_{i=1}^{\infty}$, $x,y\in X$, and a fixed constant
$1\le C_{\rho}<\infty$ satisfying the following conditions:
\begin{enumerate}[(1)]\label{rho conditions}
\item For every $x,y\in X$ with $d(x,y)\le 1$, we have for every $i\in\N$
\begin{equation}\label{eq:rho hat minorize}
\textrm{either }\
\rho_i(x,y)
\ge
C_\rho^{-1} \frac{d(x,y)^p}{r_i^{p}}\frac{\ch_{B(y,r_i)}(x)}{\mu(B(y,r_i))}
\ \textrm{ or }\  
\rho_i(x,y)
\ge
d(x,y)^p \frac{\nu_i((d(x,y),\infty))}{\mu(B(y,d(x,y)))},
\end{equation}
where $r_i\searrow 0$ and each $\nu_i$ is a positive Radon measure on $[0,\infty)$ for which
\begin{equation}\label{eq:conditions on nui}
	\liminf_{i\to\infty}\int_0^{\delta} t^{p}\,d\nu_i\ge C_{\rho}^{-1}
	\textrm{ for all }\delta>0.
\end{equation}
Also for every $x,y\in X$ with $0<d(x,y)\le 1$, we have
\begin{equation}\label{eq:rho hat majorize}
	\rho_i(x,y)
	\le \sum_{j=1}^{\infty}d_{i,j}\frac{\ch_{B(y,2^{-j+1})
			\setminus B(y,2^{-j})}(x)}{\mu(B(y,2^{-j+1}))}
\end{equation}
for numbers $d_{i,j}\ge 0$ for which  $\sum_{j=1}^{\infty} d_{i,j}\le C_\rho$.

\item For all $\delta>0$, we have
\begin{equation}\label{eq:radius one condition}
\lim_{i\to\infty}	\left(\sup_{y\in \Om}\int_{\Om\setminus B(y,\delta)}\frac{\rho_i(x,y)}{d(x,y)^p}\,d\mu(x)
+\sup_{x\in \Om}\int_{\Om\setminus B(x,\delta)}\frac{\rho_i(x,y)}{d(x,y)^p}\,d\mu(y)\right)=0.
\end{equation}
\end{enumerate}

\begin{remark}
	Conditions \eqref{eq:rho hat majorize}
	and \eqref{eq:radius one condition} will be used to prove the upper bound
	of our main Theorem \ref{thm:main}, and they are quite close to
	the Euclidean assumptions \eqref{eq:intro rhoi conditions}.
	Since we do not have as many tools at our disposal as in Euclidean spaces,
	we additionally impose the somewhat stronger conditions
	\eqref{eq:rho hat minorize} and \eqref{eq:conditions on nui};
	these will be used to prove the lower bound.
	In Section \ref{sec:quasidecreasing} we will see that these two conditions are
	also very natural.
\end{remark}

Throughout the paper, we assume that $\mu$ is doubling, but we do not always assume that $X$ satisfies
a Poincar\'e inequality.
Nonetheless, for convenience we assume that $X$ is connected, from which it follows that
$\mu(\{x\})=0$ for every $x\in X$, see e.g. \cite[Corollary 3.9]{BB}. Thus, integrating over the set where $x=y$
in the functionals that we consider does not cause any problems.\\

%\emph{Throughout this paper we assume that $X$ consists of at least two points and is complete and connected, that $\mu$ is doubling, and that $1\le p<\infty$.}

\section{Preliminary results}\label{sec:prelis}

First we note the following basic fact: for every $f\in L^p(X)$ and every ball $B(z,r)$,
using the estimate
\[
|f(x)-f(y)|^p\le 2^{p-1}(|f(x)-f_{B(z,r)}|^p+|f(y)-f_{B(z,r)}|^p),
\quad x,y\in B(z,r),
\]
we get
\begin{equation}\label{eq:Poincare other form}
\int_{B(z,r)} \int_{B(z,r)} |f(x)-f(y)|^p\,d\mu(x)\,d\mu(y)
\le 2^p\mu(B(z,r))\int_{B(z,r)} |f-f_{B(z,r)}|^p\,d\mu.
\end{equation}

The next lemma is similar to \cite[Lemma 3.1(ii)]{DiMaS}.

\begin{lemma}\label{lem:int average representation}
	For any $h(x,y)\ge 0$ that is $\mu\times\mu$-measurable in $X\times X$ and
	satisfies $h(x,y)=0$ for all $x,y\in X$ with $d(x,y)\ge \delta>0$, we have
	\begin{align*}
	\int_X \int_X h(x,y)\,d\mu(x)\,d\mu(y)
	\le C_d \int_X \frac{1}{\mu(B(z,\delta))}\iint_{B(z,2\delta)\times B(z,2\delta)} h(x,y)
	\,d\mu(x)\,d\mu(y)\,d\mu(z).
	\end{align*}
\end{lemma}

\begin{proof}
	For all $x,y\in X$ with $h(x,y)\neq 0$, we have $d(x,y)<\delta$, and then
	\begin{equation}\label{eq:ball contained in intersection}
	B(x,\delta)\subset B(x,2\delta)\cap B(y,2\delta).
	\end{equation}
	Note also that $\ch_{B(z,2\delta)}(x)$ and $\ch_{B(z,2\delta)}(y)$ are lower semicontinuous functions in the product
	space $X\times X\times X$, and so
	\[
	(x,y,z)\mapsto \frac{1}{\mu(B(z,\delta))}\ch_{B(z,2\delta)}(x)\ch_{B(z,2\delta)}(y)h(x,y)
	\]
	is $\mu\times\mu\times\mu$-measurable, and we can apply Fubini's theorem.
	We estimate
	\begin{align*}
	&\int_X \frac{1}{\mu(B(z,\delta))}\iint_{B(z,2\delta)\times B(z,2\delta)} h(x,y)\,d\mu(x)\,d\mu(y)\,d\mu(z)\\
	&\qquad =\int_X\int_X\int_X \frac{1}{\mu(B(z,\delta))}\ch_{B(z,2\delta)}(x)\ch_{B(z,2\delta)}(y)h(x,y)
	\,d\mu(x)\,d\mu(y)\,d\mu(z)\\
	&\qquad =\int_X \int_X \int_{B(x,2\delta)\cap B(y,2\delta)}
	\frac{1}{\mu(B(z,\delta))}\,d\mu(z)\,
	h(x,y)\,d\mu(x)\,d\mu(y)\quad\textrm{by Fubini}\\
	&\qquad \ge \int_X \int_X \int_{B(x,\delta)}
	\frac{1}{\mu(B(z,\delta))}\,d\mu(z)\,
	h(x,y)\,d\mu(x)\,d\mu(y)\quad\textrm{by }\eqref{eq:ball contained in intersection}\\
	&\qquad \ge \frac{1}{C_d}\int_X \int_X \int_{B(x,\delta)}
	\frac{1}{\mu(B(x,\delta))}\,d\mu(z)\,
	h(x,y)\,d\mu(x)\,d\mu(y)\quad\textrm{since }B(z,\delta)\subset B(x,2\delta)\\
	&\qquad = \frac{1}{C_d}\int_X \int_X 
	h(x,y)\,d\mu(x)\,d\mu(y).
	\end{align*}
\end{proof}

For an open set $U\subset X$ and $\delta>0$, denote
\begin{equation}\label{eq:neighborhood notation}
	U_{\delta}:=\{x\in U\colon d(x,X\setminus U)>\delta\}
	\quad\textrm{and}\quad
	U(\delta):=\{x\in X\colon d(x,U)<\delta\}.
\end{equation}

\begin{lemma}\label{lem:int average representation local}
	For any function $h(x,y)\ge 0$ that is $\mu\times\mu$-measurable on $X$ and
	satisfies $h(x,y)=0$ for all $x,y\in X$ with $d(x,y)\ge \delta>0$,
	and for an open set $U\subset X$, we have
	\begin{align*}
		&\int_{U}\int_{X} h(x,y)\,d\mu(x)\,d\mu(y)\\
		&\qquad \le C_d \int_{U(2\delta)} \frac{1}{\mu(B(z,\delta))}\iint_{B(z,2\delta)\times B(z,2\delta)} h(x,y)
		\,d\mu(x)\,d\mu(y)\,d\mu(z).
	\end{align*}
\end{lemma}
\begin{proof}
Apply Lemma \ref{lem:int average representation} with the function $h$ replaced by
$h(x,y)\ch_{U}(y)$.
\end{proof}

\section{Upper bound of Theorem \ref{thm:main}}\label{sec:one direction}

Recall that we always denote by $\Om$ an open subset of $X$, and that $1\le p<\infty$.

In order to prove the upper bound of our main Theorem \ref{thm:main}, we first prove the following result.
Recall the notation $U(R)$ from \eqref{eq:neighborhood notation}.

\begin{proposition}\label{prop:one direction}
Suppose $X$ supports the $(p,p)$-Poincar\'e inequality \eqref{eq:pp Poincare}.
Let $f\in L^p(X)$ and $0<R\le 1$, and suppose $U\subset X$ is open.
Suppose $\{\rho_i\}_{i=1}^{\infty}$ is a sequence of mollifiers that satisfy
\eqref{eq:rho hat majorize}.
Then
\begin{equation}\label{eq:upper bound with 8 lambda}
\int_{U} \int_{B(y,R)} \frac{|f(x)-f(y)|^p}{d(x,y)^p}\rho_i(x,y)\,d\mu(x)\,d\mu(y)
\le CE_{f,p}(U(8\lambda R))
\end{equation}
for every $i\in\N$ and for a constant $C=C(C_d,C_P,\lambda,C_{\rho})$.
\end{proposition}

\begin{proof}
We can assume that $E_p(f,U(8\lambda R))<\infty$.
Recall the condition \eqref{eq:rho hat majorize}.
Note that on the left-hand side of \eqref{eq:upper bound with 8 lambda}
we require $d(x,y)<R$, but we also know that
$\ch_{B(y,2^{-j+1})\setminus B(y,2^{-j})}(x)$ can be nonzero only when $2^{-j}<d(x,y)$.
Then necessarily
\begin{equation}\label{eq:j and R relation}
2^{-j} \le R.
\end{equation}
For every $j\in\Z$ satisfying \eqref{eq:j and R relation}, we estimate
\begin{equation}\label{eq:rho1 term estimate}
\begin{split}
&\int_{U}\int_{B(y,R)} \frac{|f(x)-f(y)|^p}{d(x,y)^p} \frac{\ch_{B(y,2^{-j+1})\setminus B(y,2^{-j})}(x)}{\mu(B(y,2^{-j+1}))} \,d\mu(x)\,d\mu(y)\\
&\qquad \le \int_{U}\int_{X} \frac{|f(x)-f(y)|^p}{d(x,y)^p} \frac{\ch_{B(y,2^{-j+1})\setminus B(y,2^{-j})}(x)}{\mu(B(y,2^{-j+1}))} \,d\mu(x)\,d\mu(y)\\
&\qquad \le C_d\int_{ U(2^{-j+2})} \frac{1}{\mu(B(z,2^{-j+1}))}\iint_{[B(z,2^{-j+2})]^2}
\frac{|f(x)-f(y)|^p}{d(x,y)^p}\\
&\qquad \qquad \times\frac{\ch_{X\setminus B(y,2^{-j})}(x)}{\mu(B(y,2^{-j+1}))}\,d\mu(x)\,d\mu(y)\,d\mu(z)
\quad\textrm{by Lemma }\ref{lem:int average representation local}.
\end{split}
\end{equation}
We estimate further
\begin{equation}\label{eq:rho1 term estimate 2}
\begin{split}
&\frac{1}{\mu(B(z,2^{-j+1}))}\iint_{[B(z,2^{-j+2})]^2}
\frac{|f(x)-f(y)|^p}{d(x,y)^p} \frac{\ch_{X\setminus B(y,2^{-j})}(x)}{\mu(B(y,2^{-j+1}))}\,d\mu(x)\,d\mu(y)\\
&\qquad \le 2^{jp}\frac{C_d^3}{\mu(B(z,2^{-j+2}))^2}\iint_{[B(z,2^{-j+2})]^2}
|f(x)-f(y)|^p \,d\mu(x)\,d\mu(y)\\
&\qquad \le 2^{(j+1)p}\frac{C_d^3}{\mu(B(z,2^{-j+2}))}\int_{B(z,2^{-j+2})}
|f(x)-f_{B(z,2^{-j+2})}|^p \,d\mu(x)\quad\textrm{by }\eqref{eq:Poincare other form}\\
&\qquad \le 8^p C_P C_d^3 \frac{E_{f,p}(B(z,2^{-j+2}\lambda ))}{\mu(B(z,2^{-j+2}))}
\quad \textrm{by the Poincar\'e inequality }\eqref{eq:poincare general form}.
\end{split}
\end{equation}
Denote the smallest integer at least $a\in\R$ by $\lceil a\rceil$.
Combining the above with \eqref{eq:rho1 term estimate}, we get
\begin{equation}\label{eq:estimate from above for one term}
\begin{split}
&\int_{U}\int_{B(y,R)} \frac{|f(x)-f(y)|^p}{d(x,y)^p} \frac{\ch_{B(y,2^{-j+1})
		\setminus B(y,2^{-j})}(x)}{\mu(B(y,2^{-j+1}))} \,d\mu(x)\,d\mu(y)\\
&\qquad \le 8^p C_P C_d^4 \int_{U(2^{-j+2})}\frac{E_{f,p}(B(z,2^{-j+2}
	\lambda ))}{\mu(B(z,2^{-j+2}))}\,d\mu(z)\\
&\qquad = 8^p C_P C_d^4 \int_{X}\int_{X}\ch_{U(2^{-j+2})}(z)
\frac{\ch_{B(z,2^{-j+2}
	\lambda )}(w)}{\mu(B(z,2^{-j+2}))}\,dE_{f,p}(w)\,d\mu(z)\\
&\qquad \le 8^p C_P C_d^4 \int_{X}\int_{X}\ch_{U(2^{-j+3}\lambda )}(w)
\frac{\ch_{B(w,2^{-j+2}
			\lambda )}(z)}{\mu(B(z,2^{-j+2}))}\,dE_{f,p}(w)\,d\mu(z)\\
&\qquad = 8^p C_P C_d^2 \int_{X}\int_{X}\ch_{U(2^{-j+3}\lambda )}(w)
\frac{\ch_{B(w,2^{-j+2}
		\lambda )}(z)}{\mu(B(z,2^{-j+2}))}\,d\mu(z)\,dE_{f,p}(w)\quad\textrm{by Fubini}\\
&\qquad \le 8^p C_P C_d^{3+\lceil \log_2 \lambda \rceil } \int_{X}\ch_{U(2^{-j+3}\lambda )}(w) \,dE_{f,p}(w)\\
&\qquad = 8^p C_P C_d^{3+\lceil \log_2 \lambda \rceil } E_{f,p}(U(2^{-j+3}\lambda )).
\end{split}
\end{equation}
Recalling \eqref{eq:rho hat majorize}, we get
\begin{equation}\label{eq:functional estimate 1}
\begin{split}
&\int_{U}\int_{B(y,R)} \frac{|f(x)-f(y)|^p}{d(x,y)^p}\rho_i(x,y)\,d\mu(x)\,d\mu(y)\\
&\qquad \le\sum_{j=1}^{\infty}d_{i,j}\int_{U}\int_{B(y,R)} \frac{|f(x)-f(y)|^p}{d(x,y)^p}\frac{\ch_{B(y,2^{-j+1})
		\setminus B(y,2^{-j})}(x)}{\mu(B(y,2^{-j+1}))}\,d\mu(x)\,d\mu(y)\\
&\qquad \le 8^p C_P C_d^{3+\lceil \log_2 \lambda\rceil } \sum_{j=1}^{\infty} d_{i,j}E_{f,p}(U(2^{-j+3}\lambda ))
\quad\textrm{by }\eqref{eq:estimate from above for one term}\\
&\qquad \le 8^p C_P C_d^{3+\lceil \log_2 \lambda \rceil } \sum_{j=1}^{\infty} d_{i,j}E_{f,p}(U(8\lambda R))
\quad\textrm{by }\eqref{eq:j and R relation}\\
&\qquad \le 8^p C_P C_d^{3+\lceil \log_2 \lambda \rceil } C_{\rho}E_{f,p}(U(8\lambda R))
\end{split}
\end{equation}
by the assumption $\sum_{j=1}^{\infty} d_{i,j}\le C_{\rho}$.
\end{proof}

Now we can prove one direction of our main theorem.
Recall the definition of a strong $p$-extension domain from
Definition \ref{def:strong BV extension}.

\begin{theorem}
Suppose $X$ supports a $(p,p)$--Poincar\'e inequality.
Suppose $\Om\subset X$ is a strong $p$-extension domain, and let $f\in \BV(\Om)$ if $p=1$, and
$f\in \widehat{N}^{1,p}(\Om)$ if $1<p<\infty$.
Suppose $\{\rho_i\}_{i=1}^{\infty}$ is a sequence of mollifiers that satisfy
\eqref{eq:rho hat majorize} and \eqref{eq:radius one condition}.
Then
\[
\limsup_{i\to\infty}\int_{\Om}\int_{\Om} \frac{|f(x)-f(y)|^p}{d(x,y)^p}\rho_i(x,y)\,d\mu(x)\,d\mu(y)\le CE_{f,p}(\Om)
\]
for a constant $C=C(C_d,C_P,\lambda,C_{\rho})$.
\end{theorem}

\begin{proof}
Consider $0<R\le 1$.
Recalling the notation $\Om_{8\lambda R}$ from \eqref{eq:neighborhood notation}, we have
\begin{align*}
&\int_{\Om} \int_{\Om}\frac{|f(x)-f(y)|^p}{d(x,y)^p}\rho_i(x,y)\,d\mu(x)\,d\mu(y)\\
&\qquad= \int_{\Om} \int_{\Om\setminus B(y,R)}\frac{|f(x)-f(y)|^p}{d(x,y)^p}\rho_i(x,y)\,d\mu(x)\,d\mu(y)\\
&\qquad\qquad + \int_{\Om_{8\lambda R}} \int_{ B(y,R)}\frac{|f(x)-f(y)|^p}{d(x,y)^p}\rho_i(x,y)\,d\mu(x)\,d\mu(y)\\
&\qquad\qquad\qquad + \int_{\Om\setminus \Om_{8\lambda R}}
\int_{\Om\cap B(y,R)}\frac{|f(x)-f(y)|^p}{d(x,y)^p}\rho_i(x,y)\,d\mu(x)\,d\mu(y).
\end{align*}
For the first term, we estimate 
\begin{align*}
	&\int_{\Om} \int_{\Om\setminus B(y,R)}\frac{|f(x)-f(y)|^p}{d(x,y)^p}\rho_i(x,y)\,d\mu(x)\,d\mu(y)\\
	&\qquad \le 2^{p}\int_{\Om} \int_{\Om\setminus B(y,R)}\frac{|f(x)|^p+|f(y)|^p}{d(x,y)^p}\rho_i(x,y)\,d\mu(x)\,d\mu(y)\\
	&\qquad \le 2^{p}\int_{\Om}|f(y)|^p \int_{\Om\setminus B(y,R)}\frac{\rho_i(x,y)}{d(x,y)^p}\,d\mu(x)\,d\mu(y)\\
	&\qquad \qquad +2^{p}\int_{\Om}|f(x)|^p \int_{\Om\setminus B(x,R)}\frac{\rho_i(x,y)}{d(x,y)^p}\,d\mu(y)\,d\mu(x)\quad\textrm{by Fubini}\\
	&\qquad \le 2^{p}\int_{\Om}|f|^p\,d\mu 
	\left(\sup_{y\in \Om}\int_{\Om\setminus B(y,R)}\frac{\rho_i(x,y)}{d(x,y)^p}\,d\mu(x)
	+\sup_{x\in \Om}\int_{\Om\setminus B(x,R)}\frac{\rho_i(x,y)}{d(x,y)^p}\,d\mu(y)\right)\\
	&\qquad \to 0
\end{align*}
as $i\to\infty$ by \eqref{eq:radius one condition}.

For the second term, we get from Proposition \ref{prop:one direction} that
\begin{align*}
 \int_{\Om_{8\lambda R}} \int_{B(y,R)}\frac{|f(x)-f(y)|^p}{d(x,y)^p}\rho_i(x,y)\,d\mu(x)\,d\mu(y)
  &\le CE_{f,p}(\Om_{8\lambda R}(8\lambda R))\\
  &\le CE_{f,p}(\Om).
\end{align*}

Then we estimate the third term.
Since $\Om$ is a strong $p$-extension domain,
we find an extension $F\in\BV(X)$ in the case $p=1$,
and $F\in \widehat{N}^{1,p}(X)$ in the case $1<p<\infty$,
and in both cases $E_{F,p}(\partial \Om)=0$.
Write $U:=\Om\setminus \Om_{8\lambda R}$,
and note that $U(8\lambda R)\subset \Om(8\lambda R)\setminus \Om_{16\lambda R}$.
Thus we can estimate the third term by
\begin{align*}
&\limsup_{i\to\infty}\int_{\Om\setminus \Om_{8\lambda R}}
 \int_{\Om\cap B(y,R)}\frac{|f(x)-f(y)|^p}{d(x,y)^p}\rho_i(x,y)\,d\mu(x)\,d\mu(y)\\
 &\qquad \le \limsup_{i\to\infty}\int_{U}
 \int_{B(y,R)}\frac{|F(x)-F(y)|^p}{d(x,y)^p}\rho_i(x,y)\,d\mu(x)\,d\mu(y)\\
&\qquad \le CE_{F,p}(\Om(8\lambda R)\setminus \Om_{16\lambda R})
\end{align*}
by Proposition \ref{prop:one direction}.
This goes to zero as $R\to 0$, since $E_{F,p}(\partial \Om)=0$.
Combining the three terms, we get
\[
\limsup_{i\to\infty}\int_{\Om} \int_{\Om}\frac{|f(x)-f(y)|^p}{d(x,y)^p}\rho_i(x,y)\,d\mu(x)\,d\mu(y)
\le CE_{f,p}(\Om).
\]
\end{proof}

\section{Lower bound of  Theorem \ref{thm:main}}\label{sec:other direction}

In this section we prove the lower bound of Theorem \ref{thm:main}.
Note that in this section we do not need to assume a Poincar\'e inequality.
As usual, $\Om\subset X$ is an open set.
Given a ball $B=B(x,r)$ with a specific center $x\in X$ and radius $r>0$, we denote
$2B:=B(x,2r)$.
The distance between two sets $A,D\subset X$ is denoted by
\[
\dist(A,D):=\inf\{d(x,y)\colon x\in A,\,y\in D\}.
\]

\begin{lemma}\label{lem:covering lemma}
Consider an open set $U\subset \Om$ with $\dist(U,X\setminus \Om)>0$,
and a scale $0<R<\dist(U,X\setminus \Om)/10$.
Then we can choose an at most countable covering
$\{B_j=B(x_j,R)\}_{j}$ of $U(5R)$ such that $x_j\in U(5R)$,
each ball $5B_j$ is contained in $\Om$, and the balls $\{5B_j\}_{j=1}^{\infty}$ can be divided
into at most $C_d^8$ collections of pairwise disjoint balls.
\end{lemma}
\begin{proof}
Consider a covering $\{B(x,R/5)\}_{x\in U(5R)}$.
By the $5$-covering theorem, see e.g. \cite[p. 60]{HKST},
we can choose a countable collection of disjoint balls
$\{B(x_j,R/5)\}_j$ such that the balls $B_j =B(x_j,R)$ cover $U(5R)$.
Consider a ball $B_j$ and denote by $I_j$ those $k\in \N$ for which $5 B_k\cap 5B_j\neq \emptyset$.
Then
\begin{align*}
\sum_{k\in I_j}\mu(B_k)
\le C_d^3\sum_{k\in I_j}\mu(\tfrac 15 B_k)
\le C_d^3\mu(11B_j)
&\le C_d^3\mu(21B_l)\quad\textrm{for any }l\in I_j\\
&\le C_d^8 \mu(B_l),
\end{align*}
and so $I_j$ has cardinality at most $C_d^8$.
We can recursively choose maximal collections of pairwise disjoint balls $5B_j$.
After at most $C_d^8$ steps, we have exhausted all of the balls $5B_j$.
\end{proof}

Let $C_0:=3C_d^8$.
Given such a covering of $U(5R)$, 
we can take a partition of unity $\{\phi_j\}_{j=1}^{\infty}$ subordinate to the
covering, such that $0\le \phi_j\le 1$,
\begin{equation}\label{eq:Lipschitz function}
	\textrm{each } \phi_j  \textrm{ is a }C_0/R\textrm{-Lipschitz function},
	\end{equation}
	and $\supp(\phi_j)\subset 2B_j$ for each 
$j\in\N$;
see e.g. \cite[p. 104]{HKST}.
Finally, we can define a \emph{discrete convolution} $h$ of 
any $f\in L^1(\Om)$ with respect to the covering by
\[
h:=\sum_{j}f_{B_j}\phi_j.
\]
Clearly $h\in \Lip_{\loc}(U)$.

\begin{theorem}\label{thm:lower bound}
	Suppose $\rho_i$ is a sequence of mollifiers satisfying
	\eqref{eq:rho hat minorize}.
	Suppose $f\in L^p(\Om)$. Then
	\begin{equation}\label{eq:main theorem equation lower}
			C_1 E_{f,p}(\Om)
			\le \liminf_{i\to\infty}\int_{\Om}\int_{\Om} \frac{|f(x)-f(y)|^p}{d(x,y)^p}\rho_i(x,y)\,d\mu(x)\,d\mu(y)
	\end{equation}
	for some constant $C_1$ depending only on $C_{\rho}$ and on the doubling constant of the measure.
\end{theorem}
Note that here we do not impose any conditions on the open set $\Om\subset X$.

\begin{proof}
We can assume that
\[
\liminf_{i\to\infty}\int_{\Om}\int_{\Om} \frac{|f(x)-f(y)|^p}{d(x,y)^p}\rho_i(x,y)\,d\mu(x)\,d\mu(y)=:M<\infty.
\]
Fix $0<\eps<1$. Passing to a subsequence (not relabeled), we can assume that
\[
\liminf_{i\to\infty}\int_{\Om}\int_{\Om}
|f(x)-f(y)|^p\rho_i(x,y)\,d\mu(x)\,d\mu(y)\le M+\eps \quad\textrm{for all }i\in\N.
\]
Assuming the second option of \eqref{eq:rho hat minorize}, we get
\[
\int_{\Om}\int_{\Om}
|f(x)-f(y)|^p\frac{\nu_i((d(x,y),\infty))}{\mu(B(y,d(x,y)))}\ch_{B(y,1)}(x)\,d\mu(x)\,d\mu(y)\le M+\eps.
\]
It follows that
\begin{align*}
M +\eps&\ge  \int_{\Om}\int_{\Om}\int_{d(x,y)}^{\infty}
|f(x)-f(y)|^p\frac{1}{\mu(B(y,d(x,y)))}\ch_{B(y,1)}(x)\,d\nu_i(t)\,d\mu(x)\,d\mu(y)\\
&= 
\int_{0}^{\infty}\iint_{\{x,y\in\Om\colon d(x,y)< t\}}
\frac{|f(x)-f(y)|^p}{t^p}\frac{\ch_{B(y,1)}(x)}{\mu(B(y,d(x,y)))}\,d\mu(x)\,d\mu(y)\,t^p\,d\nu_i(t)
\end{align*}
by Fubini's theorem.
By \eqref{eq:conditions on nui},
given an arbitrarily small $0<\delta<1$,
we have
\[
\liminf_{i\to\infty}\int_0^{\delta} t^{p}\,d\nu_i\ge C_{\rho}^{-1},
\]
and so
$\int_0^{\delta} t^{p}\,d\nu_i\ge (1-\eps)C_{\rho}^{-1}$ for all sufficiently large $i\in\N$.
Then there necessarily exists $0<t\le \delta$ such that
\[
\iint_{\{x,y\in\Om\colon d(x,y)< t\}}
\frac{|f(x)-f(y)|^p}{t^p}\frac{\ch_{B(y,1)}(x)}{\mu(B(y,d(x,y)))}\,d\mu(x)\,d\mu(y)
\le \frac{M +\eps}{(1-\eps)C_{\rho}^{-1}}.
\]
In other words, we find arbitarily small $t>0$ such that
\begin{equation}\label{eq:t arbitrarily small}
\int_{\Om}\int_{\Om} \frac{|f(x)-f(y)|^p}{t^p}
\frac{\ch_{B(y,t)\cap \Om}(x)}{\mu(B(y,t))}\,d\mu(x)\,d\mu(y)
\le \frac{(M+\eps)C_{\rho}}{1-\eps}.
\end{equation}
We obviously obtain this also if the first option of \eqref{eq:rho hat minorize} holds.
Fix a small $t>0$.
Let $U\subset \Om$ with $\dist(U,X\setminus \Om)>t$, and
let $R:=t/10$.
Consider a covering $\{B_j\}_{j=1}^{\infty}$ of $U(5R)$ at scale $R>0$,
as described in Lemma \ref{lem:covering lemma}. Then consider the discrete convolution
\[
h:=\sum_{j}f_{B_j}\phi_j.
\]
We define the pointwise asymptotic Lipschitz number by
\[
\Lip_h(x):=\limsup_{r\to 0}\frac{\sup_{y\in  B(x,r)}|h(y)-h(x)|}{r},\quad x\in U.
\]
Suppose $x\in U$. Then $x\in B_j$ for some $j\in\N$. Consider any other point $y\in B_j$.
Denote by $I_j$ those $k\in\N$ for which $2B_k \cap 2B_j\neq \emptyset$.
We estimate
\begin{equation}\label{eq:upper gradient of discrete convolution}
	\begin{split}
|h(x)-h(y)|
& = \left|\sum_{k\in I_j}f_{B_k}(\phi_k(x)-\phi_k(y))\right|\\
& = \left|\sum_{k\in I_j}(f_{B_k}-f_{B_j})(\phi_k(x)-\phi_k(y))\right|\\
&\le \frac{C_0d(x,y)}{R}\left(\sum_{k\in I_j}\,\vint{B_k}|f-f_{5B_j}|\,d\mu
+\sum_{k\in I_j}\,\vint{B_j}|f-f_{5B_j}|\,d\mu\right)\quad\textrm{by }\eqref{eq:Lipschitz function}\\
&\le \frac{2C_0C_d^3d(x,y)}{R}\sum_{k\in I_j}\,\vint{5B_j}|f-f_{5B_j}|\,d\mu\\
&\le \frac{2C_0^2 C_d^3d(x,y)}{R}\,\vint{5B_j}\,\vint{5B_j}|f(z)-f(w)|\,d\mu(z)\,d\mu(w),
	\end{split}
\end{equation}
since by Lemma \ref{lem:covering lemma} we know that $I_j$ has cardinality at most $C_0$.
Letting $y\to x$, we obtain an estimate for $\Lip_h$ in the ball $B_j$.
In total, we conclude (we track the constants for a while in order to make the estimates more explicit)
\[
\Lip_h
\le \frac{2C_0^2 C_d^3}{R}\sum_{j}\,\ch_{B_j}\vint{5B_j}\,\vint{5B_j}|f(x)-f(y)|\,d\mu(x)\,d\mu(y).
\]
Since the balls $\{B_j\}_{j}$ can be divided
into at most $C_0$ collections of pairwise disjoint balls, we get
\begin{equation}\label{eq:upper gradient of discrete convolution}
	\begin{split}
(\Lip_h)^p
&\le \frac{(2C_0^2 C_d^3)^p C_0^p}{R^p}
\sum_{j}\ch_{B_j}\left(\vint{5B_j}\,\vint{5B_j}|f(x)-f(y)|^p\,d\mu(x)\,d\mu(y)\right)^p\\ 
&\le \frac{(2C_0^2 C_d^3)^p C_0^p}{R^p}
\sum_{j}\ch_{B_j}\vint{5B_j}\,\vint{5B_j}|f(x)-f(y)|^p\,d\mu(x)\,d\mu(y)
\quad\textrm{by H\"older}\\ 
&\le \frac{(2C_0^2 C_d^3)^p (10C_0)^p C_d^2}{(10R)^p}
\sum_{j}\ch_{B_j}\vint{5B_j}\int_{5B_j}|f(x)-f(y)|^p\frac{\ch_{B(y,10R)}(x)}
{\mu(B(y,10R))}\,d\mu(x)\,d\mu(y).
	\end{split}
\end{equation}
Thus
\begin{equation}\label{eq:Lip h estimate}
\begin{split}
\int_U (\Lip_h)^p\,d\mu
&\le \frac{(2C_0^2 C_d^3)^p (10C_0)^p C_d^2}{(10R)^p }
\sum_{j}\int_{5B_j}\int_{5B_j}|f(x)-f(y)|^p\frac{\ch_{B(y,10R)}(x)}
{\mu(B(y,10R))}\,d\mu(x)\,d\mu(y)\\
&\le \frac{(2C_0^2 C_d^3)^p (10C_0)^p C_d^2 C_0}{(10R)^p}\int_{\Om}\int_{\Om}|f(x)-f(y)|^p\frac{\ch_{B(y,10R)\cap \Om}(x)}
{\mu(B(y,10R))}\,d\mu(x)\,d\mu(y)\\
&\le C\frac{(M+\eps)C_{\rho}}{(1-\eps)}
\end{split}
\end{equation}
by \eqref{eq:t arbitrarily small}, with $C:=(2C_0^2 C_d^3)^p (10C_0)^pC_d^2 C_0$.
We know that the minimal $p$-weak upper gradient $g_{h}$ of $h$ in $U$ satisfies $g_{h}\le \Lip_h$ $\mu$-a.e. in $U$,
 see e.g. \cite[Proposition 1.14]{BB}.
 
Recall that we can do the above for arbitrarily small $t>0$ and thus arbitrarily small $R>0$.
From now on, we can consider any open $U\subset \Om$ with $\dist(U,X\setminus \Om)>0$.
We get a sequence of discrete convolutions $\{h_i\}_{i=1}^{\infty}$ corresponding
to scales $R_i\searrow 0$, such that $\{g_{h_i}\}_{i=1}^{\infty}$ is a bounded sequence in
$L^p(U)$. From the properties of discrete convolutions, see e.g. \cite[Lemma 5.3]{HKT},
we know that $h_i\to f$ in $L^p(U)$.
Passing to a subsequence (not relabeled), we also have $h_i(x)\to f(x)$ for $\mu$-a.e. $x\in U$.
When $p=1$, we get
\[
\Vert Df\Vert(U)
\le \liminf_{i\to\infty}\int_U g_{h_i}\,d\mu
\le \liminf_{i\to\infty}\int_U \Lip_{h_i}\,d\mu
\le  C\frac{(M+\eps)C_{\rho}}{(1-\eps)},
\]
and so $f\in \BV(U)$.
In the case $1<p<\infty$,
by reflexivity of the space $L^p(U)$, we find a subsequence of $\{h_i\}_{i=1}^{\infty}$ (not relabeled)
and $g\in L^p(U)$ such that $g_{h_i}\to g$ weakly in $L^p(U)$ (see e.g. \cite[Section 2]{HKST}).
By Mazur's lemma (Theorem \ref{thm:Mazur lemma}), for suitable convex combinations we get the
strong convergence $\sum_{l=i}^{N_i}a_{i,l} g_{h_l}\to g$ in $L^p(U)$.
We still have $\sum_{l=i}^{N_i}a_{i,l} h_l(x)\to f(x)$ for $\mu$-a.e. $x\in U$.
Define
\[
\widetilde{h}(x):=
\limsup_{i\to\infty}\sum_{l=i}^{N_i}a_{i,l} h_l(x),\quad x\in U.
\]
Then $\widetilde{h}=f$ $\mu$-a.e. in $U$.
Denote $N:=\{x\in U\colon |\widetilde{h}(x)|<\infty\}$, so that $\mu(N)=0$.
For $p$-a.e. curve $\gamma$ in $U$, denoting the end points by $x,y$, we have that either
$x\notin N$ or $y\notin N$; see \cite[Corollary 1.51]{BB}.
For such $\gamma$, we obtain
\[
|\widetilde{h}(x)-\widetilde{h}(y)|
\le \limsup_{i\to\infty}\left|\sum_{l=i}^{N_i}a_{i,l} h_l(x)-\sum_{l=i}^{N_i}a_{i,l} h_l(y)\right|
\le \limsup_{i\to\infty}\int_{\gamma}\sum_{l=i}^{N_i}a_{i,l} g_{h_l}\,ds
= \int_{\gamma}g\,ds
\]
by Fuglede's lemma (Lemma \ref{lem:Fuglede lemma}), exclusing another curve family of zero $p$-modulus.
Hence $g$ is a $p$-weak upper gradient of $\widetilde{h}$ in $U$, and so
for the minimal $p$-weak upper gradient we have
\[
\int_U g_{\widetilde{h}}^p\,d\mu
\le \int_U g^p\,d\mu
\le \limsup_{i\to\infty}\int_U g_{h_i}^p\,d\mu
\le C\frac{(M+\eps)C_{\rho}}{(1-\eps)}
\]
by \eqref{eq:Lip h estimate}.
Since $f=\widetilde{h}$ $\mu$-a.e. in $U$, we have $f\in \widehat{N}^{1,p}(U)$.
Note that now $E_{f,p}$ is a Radon measure on $\Om$.
Exhausting $\Om$ by sets $U$, in both cases we obtain
\begin{align*}
E_{f,p}(\Om)
&\le C\frac{(M+\eps)C_{\rho}}{(1-\eps)}\\
&=\frac{CC_{\rho}}{1-\eps}\liminf_{i\to\infty}
\left[\int_{\Om}
\int_{\Om} \frac{|f(x)-f(y)|^p}{d(x,y)^p}\rho_i(x,y)\,d\mu(x)\,d\mu(y)+\eps\right].
\end{align*}
Letting $\eps\to 0$, this proves \eqref{eq:main theorem equation lower}.
\end{proof}

\section{Corollaries}\label{sec:quasidecreasing}

The conditions \eqref{eq:rho hat minorize}--\eqref{eq:radius one condition} that we impose on the
mollifiers $\rho_i$ are quite flexible, and so we can obtain various existing results in the
literature as special cases of our main Theorem \ref{thm:main}.
The following is essentially \cite[Theorem 1.4]{DiMaS}, except that we consider an open set $\Om$
instead of the whole space $X$.

\begin{corollary}\label{cor:DiMarino}
	Suppose $X$ supports a $(p,p)$--Poincar\'e inequality.
	Let $\Om\subset X$ be a strong $p$-extension domain, and let $f\in L^p(\Om)$. Then
	\begin{equation}
		\begin{split}
			&C^{-1}E_{f,p}(\Om)
			\le \liminf_{s\nearrow 1}(1-s)\int_{\Om}\int_{\Om}\frac{|f(x)-f(y)|^p}{d(x,y)^{ps}\mu(B(y,d(x,y)))}\,d\mu(x)\,d\mu(y)\\
			&\qquad \le \limsup_{s\nearrow 1}(1-s)\int_{\Om}\int_{\Om}\frac{|f(x)-f(y)|^p}{d(x,y)^{ps}\mu(B(y,d(x,y)))}\,d\mu(x)\,d\mu(y)
			\le CE_{f,p}(\Om)
		\end{split}
	\end{equation}
	for some constant $C\ge 1$ depending only on $p$, the doubling constant of the measure, and the constants in the
	Poincar\'e inequality.
\end{corollary}
\begin{proof}
	This is obtained from Theorem \ref{thm:main} with the choice of mollifiers
	\[
	\rho_i(x,y):=(1-s_i)\frac{1}{d(x,y)^{p(s_i-1)}\mu(B(y,d(x,y)))},\quad x,y\in X,
	\]
	where $s_i\nearrow 1$ as $i\to\infty$.
	We only need to check that conditions \eqref{eq:rho hat minorize}--\eqref{eq:radius one condition}
	are satisfied.
	We have
	\[
	\rho_i(x,y)=d(x,y)^p\frac{\nu_i((d(x,y),\infty))}{\mu(B(y,d(x,y)))}
	\quad \textrm{with}\quad
	d\nu_i(t):=ps_i(1-s_i)t^{-ps_i-1}\,dt,
	\]
	and so
	\[
	\liminf_{i\to\infty}\int_{0}^{\delta}t^p\,d\nu_i
	=p\liminf_{i\to\infty}s_i(1-s_i)\int_{0}^{\delta}t^{-p(s_i-1)-1}\,dt
	=1\quad\textrm{for all }\delta>0,
	\]
	satisfying the second option of \eqref{eq:rho hat minorize}, and 
	\eqref{eq:conditions on nui}.
	
	For every $x,y\in X$ with $0<d(x,y)\le 1$, we have
	\[
	(1-s_i)\frac{1}{d(x,y)^{p(s_i-1)}\mu(B(y,d(x,y)))}
	\le \sum_{j=1}^{\infty}d_{i,j} \frac{\ch_{B(y,2^{-j+1})
			\setminus B(y,2^{-j})}(x)}{\mu(B(y,2^{-j+1}))}
	\]
	with $d_{i,j}=C_d(1-s_i)2^{(-j+1)p(1-s_i)}$.
	Here
	\[
	\sum_{j=1}^{\infty}2^{(-j+1)p(1-s_i)}
	\le 2\int_0^{2} t^{p(1-s_i)-1}\,dt
	=\frac{2^{1+p(1-s_i)}}{p(1-s_i)}
	\le \frac{2^{1+p}}{p(1-s_i)},
	\]
	and so
	\[
	\sum_{j=0}^{\infty}d_{i,j}\le \frac{2^{1+p}C_d}{p},
	\]
	satisfying \eqref{eq:rho hat majorize}.
	Finally, we estimate
	\begin{align*}
	\frac{\rho_i(x,y)}{d(x,y)^p}
	&=(1-s_i)\frac{1}{d(x,y)^{p s_i}\mu(B(y,d(x,y)))}\\
	&\le C_d(1-s_i)\sum_{j\in \Z}2^{jps_i} \frac{\ch_{B(y,2^{-j+1})
			\setminus B(y,2^{-j})}(x)}{\mu(B(y,2^{-j+1}))},
	\end{align*}
	and so (the notation $\sum_{j\le  -\log_2 \delta}$ means that we sum over
	integers $j$ at most $-\log_2 \delta$)
	\begin{align*}
	\int_{X\setminus B(y,\delta)}\frac{\rho_i(x,y)}{d(x,y)^p}\,d\mu(x)
	&\le C_d(1-s_i) \sum_{j\le  -\log_2 \delta}2^{jps_i}\\
	&\le C_d (1-s_i)\frac{\delta^{-ps_i}}{1-2^{-ps_i}}\\
	&\to 0\quad\textrm{as }i\to\infty,
	\end{align*}
	and so \eqref{eq:radius one condition} holds.
\end{proof}

In particular, in the Euclidean setting, the functional considered in the
Corollary \ref{cor:DiMarino} reduces
to the fractional Sobolev seminorm
\[
\int_{\Om}\int_{\Om}\frac{|f(x)-f(y)|^p}{|x-y|^{N+sp}}\,dx\,dy.
\]

The functional appearing in the following corollary was previously considered by
Marola--Miranda--Shanmugalingam \cite{MMS} as well as G\'orny \cite{Gor} and Han--Pinamonti \cite{HP}
\begin{corollary}\label{cor:Gorny}
	Suppose $X$ supports a $(p,p)$--Poincar\'e inequality, and let $f\in L^p(X)$. Then
	\begin{equation}
		\begin{split}
			C^{-1}E_{f,p}(X)
			&\le \liminf_{r\searrow 0}\frac{1}{r^p}\int_X \vint{B(y,r)}|f(x)-f(y)|^p\,d\mu(x)\,d\mu(y)\\
			&\le \limsup_{r\searrow 0}\frac{1}{r^p}\int_X \vint{B(y,r)}|f(x)-f(y)|^p\,d\mu(x)\,d\mu(y)
			\le CE_{f,p}(X)
		\end{split}
	\end{equation}
	for some constant $C$ depending only on $p$, the doubling constant of the measure, and the constants in the
	Poincar\'e inequality.
\end{corollary}

\begin{proof}
	This is obtained from Theorem \ref{thm:main} with the choice
	\[
	\rho_{i}(x,y)
	=r_i^{-p}d(x,y)^p\frac{\ch_{B(y,r_i)}(x)}{\mu(B(y,r_i))},
	\]
	where $r_i\searrow 0$ as $i\to\infty$.
	Now the first option of \eqref{eq:rho hat minorize} holds.
	For every $x,y\in X$ with $0<d(x,y)\le 1$, we have
	\[
	\rho_{i}(x,y)
	\le \sum_{j=1}^{\infty}d_{i,j}\frac{\ch_{B(y,2^{-j+1})
			\setminus B(y,2^{-j})}(x)}{\mu(B(y,2^{-j+1}))}
	\]
	with $d_{i,j}= r_i^{-p}2^{(-j+1)p}\mu(B(y,2^{-j+1})) \mu(B(y,r_i))^{-1}$
	for $j\ge -\log_2 r_i$, and $d_{i,j}=0$ otherwise.
	Now
	\begin{align*}
	\sum_{j=1}^{\infty}d_{i,j}
	&= r_i^{-p} \sum_{j\ge -\log_2 r_i}  2^{(-j+1)p}\mu(B(y,2^{-j+1})) \mu(B(y,r_i))^{-1}\\
	&\le  C_d r_i^{-p} \sum_{j\ge -\log_2 r_i}  2^{(-j+1)p}\\
	&\le  2^p C_d,
	\end{align*}
	and thus \eqref{eq:rho hat majorize} is satisfied.
	The condition \eqref{eq:radius one condition} obviously holds.
\end{proof}

The following simple choice of mollifiers, considered in the Euclidean setting e.g. by Brezis
\cite[Eq. (45)]{Bre}, is also natural.
This will be used also in a counterexample in the last section.
\begin{corollary}\label{cor:Brezis}
	Suppose $X$ supports a $(p,p)$--Poincar\'e inequality. Let $f\in L^p(X)$. Then
	\begin{equation}
		\begin{split}
			&C^{-1}E_{f,p}(X)
			\le \liminf_{r\searrow 0}\int_X \vint{B(y,r)}\frac{|f(x)-f(y)|^p}{d(x,y)^p}\,d\mu(x)\,d\mu(y)\\
			&\qquad \le \limsup_{r\searrow 0}\int_X \vint{B(y,r)}\frac{|f(x)-f(y)|^p}{d(x,y)^p}\,d\mu(x)\,d\mu(y)
			\le CE_{f,p}(X).
		\end{split}
	\end{equation}
	for some constant $C$ depending only on $p$, the doubling constant of the measure, and the constants in the Poincar\'e inequality.
\end{corollary}
\begin{proof}
	This is obtained from Theorem \ref{thm:main} with the choice
	\[
	\rho_{i}(x,y)=\frac{\ch_{B(y,r_i)}(x)}{\mu(B(y,r_i))},
	\]
	where $r_i\searrow 0$ as $i\to\infty$.
	Again the first option of \eqref{eq:rho hat minorize} holds.
	We can assume that $r_i<\min\{1,\diam X/4\}$ for all $i\in\N$.
	For every $x,y\in X$ with $0<d(x,y)\le 1$, we have
	\[
	\rho_{i}(x,y)
	\le \sum_{j=1}^{\infty}d_{i,j}\frac{\ch_{B(y,2^{-j+1})\setminus B(y,2^{-j})}(x)}{\mu(B(y,2^{-j+1}))},
	\]
	where $d_{i,j}= \mu(B(y,2^{-j+1})) \mu(B(y,r_i))^{-1}$ for $j\ge -\log_2 r_i$ and $d_{i,j}=0$
	otherwise.
	
	Since $X$ is connected, there exists $z\in \partial B(y, \frac{3}{2}\cdot 2^{-j})$ for all
	$j\ge -\log_2 r_i$, and so
	\[
	B(z, 2^{-j-1})\subset B(y,2^{-j+1})\setminus B(y,2^{-j})\quad\textrm{and}\quad B(y,2^{-j+1})\subset B(z,2^{-j+2}).
	\]
	It follows that
	\[
	\mu({B(y,2^{-j+1})\setminus B(y, 2^{-j})})\ge C_d^{-3}\mu(B(y,2^{-j+1})),
	\]
	and so
	
	\begin{align*}
	\sum_{j=1}^{\infty}d_{i,j}
	&=\mu(B(y,r_i))^{-1} \sum_{j\ge -\log_2 r_i} \mu(B(y,2^{-j+1}))\\
	&\le C_d^3\mu(B(y,r_i))^{-1} \sum_{j\ge -\log_2 r_i} \mu({B(y,2^{-j+1})\setminus B(y, 2^{-j})})\\
	&\le C_d^3\mu(B(y,r_i))^{-1} \mu(B(y,2r_i))\\
	&\le C_d^4,
	\end{align*}
	satisfying \eqref{eq:rho hat majorize}.
	The condition \eqref{eq:radius one condition} obviously holds.
\end{proof}

\section{A counterexample}\label{sec:counterexample}

Recall that the conclusion of our main Theorem \ref{thm:main} has the form
\begin{equation}\label{eq:recite main result}
\begin{split}
C_1 E_{f,p}(\Om)
&\le \liminf_{i\to\infty}\int_{\Om}\int_{\Om} \frac{|f(x)-f(y)|^p}{d(x,y)^p}\rho_i(x,y)\,d\mu(x)\,d\mu(y)\\
&\le \limsup_{i\to\infty}\int_{\Om}\int_{\Om} \frac{|f(x)-f(y)|^p}{d(x,y)^p}\rho_i(x,y)\,d\mu(x)\,d\mu(y)
\le C_2 E_{f,p}(\Om).
\end{split}
\end{equation}
One natural question to ask is whether $C_1=C_2$ might hold.
G\'orny \cite{Gor} shows that with a suitable choice of the mollifiers
$\rho_i$, this holds in the case $1<p<\infty$ if we additionally assume that
at $\mu$-a.e. point $x\in X$, the tangent space is Euclidean with fixed dimension.
On the other hand, he gives an example where the dimension of the tangent space takes two different values
in two different parts of the space,
and then it is necessary to choose $C_1<C_2$.
In the example below, inspired by \cite[Example 4.8]{HKLL},
it is easy to check that the tangent space of $X$ is Euclidean with dimension $1$ at $\mu$-a.e. $x\in X$
(see definitions in \cite{Gor}), but nonetheless
we show in the case $p=1$ that $C_1<C_2$.

First consider the real line equipped with the Euclidean metric and the one-dimensional Lebesgue
measure $\mathcal L^1$.
Similarly to Corollary \ref{cor:Brezis}, consider the sequence of mollifiers
\[
\rho_i(x,y):=\frac{\ch_{[-1/i,1/i]}(|x-y|)}{2/i},\quad x,y\in \R,\quad i\in\N.
\]
From the Euclidean theory, see D\'avila \cite[Theorem 1.1]{Dav},
we know that for every $f\in L^1(\R)$, we have
\begin{equation}\label{eq:Euclidean 1d result}
\lim_{i\to\infty}\int_{\R}\int_{\R} \frac{|f(x)-f(y)|}{|x-y|}\rho_i(x,y)\,
d\mathcal L^1(x)\,d\mathcal L^1(y)
=\Vert Df\Vert(\R).
\end{equation}

\begin{example}
	Consider the space $X=[0,1]$, equipped with the Euclidean metric and a weighted measure
	$\mu$ that we will next define.
	First we construct a fat Cantor set $A$ as follows.
	Let $A_0:=[0,1]$.
	Then in each step
	$i\in\N$, we remove from $A_{i-1}$ the set $D_i$, which consists of $2^{i-1}$ open intervals of length $2^{-2i}$,
	centered at the middle points of the intervals that make up $A_{i-1}$.
	We denote $L_i:=\mathcal L^1(A_i)$,
	and we let $A=\bigcap_{i=1}^{\infty}A_i$. Then we have
	\[
	L:=\mathcal L^1(A) =\lim_{i\to\infty}L_i=1/2.
	\]
	Then define the weight
	\[
	w:=
	\begin{cases}
	2\quad\textrm{in }A,\\
	1\quad\textrm{in }X\setminus A,
	\end{cases}
	\]
	and equip the space $X$ with the weighted Lebesgue measure $d\mu:=w\,d\mathcal L^1$.
	Obviously the measure is doubling, and $X$ supports a $(1,1)$--Poincar\'e inequality.
	
	Let
	\[
	g:=2\ch_{A} \quad\textrm{and}\quad g_i=\frac{1}{L_{i-1}-L_i}\ch_{D_i},\ \ i\in\N.
	\]
	Then
	\[
	\int_0^1 g(s)\,ds= \int_0^1 g_i(s)\,ds=1\quad\textrm{for all }i\in\N.
	\]
	Next define the function
	\[
	f(x)=\int_0^x g(s)\, ds,\quad x\in [0,1].
	\]
	Now $f\in \Lip(X)$, since $g$ is bounded.
	Approximate $f$ with the functions
	\[
	f_i(x)=\int_0^x g_i(s)\, ds,\quad x\in [0,1],\quad i\in\N.
	\]
	Now also $f_i\in \Lip(X)$, and $f_i\to f$ uniformly.
	This can be seen as follows. Given $i\in\N$, the set $A_i$ consists of $2^i$ intervals of length
	$L_i/2^i$. If $I$ is one of these intervals, we have
	\[
	2^{-i}=\int_I g(s) \,ds =\int_I g_{i+1}(s) \,ds,
	\]
	and also
	\[
	\int_{X\setminus A_i}g\,d\mathcal L^1 =0=\int_{X\setminus A_i}g_{i+1}\,d\mathcal L^1 .
	\]
	Hence $f_{i+1}=f$ in  $X\setminus A_i$, and elsewhere $|f_{i+1}-f|$ is at most $2^{-i}$.
	In particular, $f_i\to f$ in $L^1(X)$ and so
	\[
	\Vert Df\Vert(X)\le \lim_{i\to\infty} \int_0^1 g_i\,d\mu
	= \lim_{i\to\infty} \int_0^1 g_i\,d\mathcal L^1=1.
	\]
	For a.e. $x\in A$, $f$ is differentiable at $x$ and so we have
	\[
	\lim_{i\to\infty}\int_X \frac{|f(x)-f(y)|}{|x-y|}\rho_i(x,y)\,d\mathcal L^1(x)=|f'(y)|.
	\]
	Thus
	\begin{align*}
	&\liminf_{i\to\infty}\int_X \int_X\frac{|f(x)-f(y)|}{|x-y|}\rho_i(x,y)\,d\mu(x)\,d\mu(y)\\
	&\qquad \ge 2 \liminf_{i\to\infty}\int_A \int_X\frac{|f(x)-f(y)|}{|x-y|}\rho_i(x,y)\,d\mathcal L^1(x)\,d\mathcal L^1(y)\\
	&\qquad \ge 2\int_A \liminf_{i\to\infty} \int_X\frac{|f(x)-f(y)|}{|x-y|}\rho_i(x,y)\,d\mathcal L^1(x)\,d\mathcal L^1(y)\quad\textrm{by Fatou}\\
	&\qquad= 2\int_X |f'(y)|\,d\mathcal L^1(y)\\
	&\qquad=2.
	\end{align*}
We conclude that
\begin{equation}\label{eq:comparison with constant 2}
\liminf_{i\to\infty}\int_X \int_X\frac{|f(x)-f(y)|}{|x-y|}\rho_i(x,y)\,d\mu(y)\,d\mu(x)
\ge 2\Vert Df\Vert(X).
\end{equation}
On the other hand, consider any nonzero Lipschitz function $f_0$ supported in $(3/8,5/8)$.
For such a function, from \eqref{eq:Euclidean 1d result} we have
\[
	\lim_{i\to\infty}\int_{X}\int_{X} \frac{|f_0(x)-f_0(y)|}{|x-y|}\rho_i(x,y)\,d\mu(x)\,d\mu(y)
	=\Vert Df_0\Vert (X),
\]
since both sides are equal to the classical quantities, that is,
the quantities obtained when the measure $\mu$ is $\mathcal L^1$.
This combined with \eqref{eq:comparison with constant 2} shows that we cannot have
$C_1=C_2$ in \eqref{eq:recite main result}.
\end{example}

\end{document}